\documentclass[12pt]{article}%
\usepackage{amsfonts}
\usepackage{amssymb}
\usepackage{amsmath}
\usepackage{amsthm}
\usepackage{graphicx}
\usepackage{ifthen}%
\renewcommand{\footnote}[1]{}
\renewcommand{\thanks}[1]{}
\newtheorem{thm}{Theorem}[section]

\newtheorem{lm}[thm]{Lemma}
\theoremstyle{remark}
\newtheorem{remark}{Remark}
\newcommand{\nc}{\newcommand}
\newcommand{\df}{\def}
\RequirePackage{ifthen}
\df\ie#1{{\it \mbox{i.\hspace{0.06em}e.\ifthenelse{\equal{#1}{,}}{\hspace{0.06em},}{\ #1}}}}
\df\eg#1{{\it \mbox{e.\hspace{0.06em}g.\ifthenelse{\equal{#1}{,}}{\hspace{0.06em},}{\ #1}}}}
\nc{\cit}[2][]{\ifthenelse{\equal{#1}{}}{\rm \cite{#2}}{\rm \cite[#1]{#2}}}
\makeatletter
\newlength{\JR@setLa}
\newlength{\JR@setLb}
\newlength{\JR@setLc}
\newlength{\JR@setLd}
\newlength{\JR@setLe}
\nc{\JR@setMta}{}
\nc{\JR@setMtb}{}
\nc{\JR@setMtc}{}
\nc{\JR@setMtd}{}
\nc{\JR@setMt}[1]{
\ifthenelse{#1=1}{\JR@setMta}{}
\ifthenelse{#1=2}{\JR@setMtb}{}
\ifthenelse{#1=3}{\JR@setMtc}{}
\ifthenelse{#1=4}{\JR@setMtd}{}
}
\nc{\JR@max}[3]{\ifthenelse{\lengthtest{\the#1>\the#2}}
{\setlength{#3}{#1}}
{\setlength{#3}{#2}}}
\nc{\JR@setds}{}
\nc{\JR@setSD}[4][]{
\ifthenelse{\equal{#1}{}}
{\renewcommand{\JR@setds}{}}
{\renewcommand{\JR@setds}{\displaystyle}}
\settoheight{\JR@setLa}{\ensuremath{\JR@setds{#2}}}
\settoheight{\JR@setLb}{\ensuremath{\JR@setds{#3}}}
\settodepth{\JR@setLc}{\ensuremath{\JR@setds{#2}}}
\settodepth{\JR@setLd}{\ensuremath{\JR@setds{#3}}}
\settowidth{\JR@setLe}{\ensuremath{\JR@setds{\left.\right.}}}
\JR@max{\JR@setLc}{\JR@setLd}{\JR@setLc}
\JR@max{\JR@setLa}{\JR@setLb}{\JR@setLa}
\addtolength{\JR@setLa}{\JR@setLc}
\ifthenelse{#4=1}{\renewcommand{\JR@setMta}{\rule[-\JR@setLc]{0pt}{\JR@setLa}}}{}
\ifthenelse{#4=2}{\renewcommand{\JR@setMtb}{\rule[-\JR@setLc]{0pt}{\JR@setLa}}}{}
\ifthenelse{#4=3}{\renewcommand{\JR@setMtc}{\rule[-\JR@setLc]{0pt}{\JR@setLa}}}{}
\ifthenelse{#4=4}{\renewcommand{\JR@setMtd}{\rule[-\JR@setLc]{0pt}{\JR@setLa}}}{}
}
\df\SetEx[#1,#2](:#3|#4:){
\JR@setSD[#2]{#3}{#4}{#1}
\left\{\JR@setMt{#1} #3\right.\left|\rule{.9\JR@setLe}{0pt}#4 \JR@setMt{#1}\right\}
}
\nc{\set}[1][1]{\SetEx[#1,ds]}
\nc{\seti}[1][1]{\SetEx[#1,]}
\makeatother
\nc{\M}{\mathbb{M}}
\begin{document}

\author{Jin-ichi Itoh
\and Jo\"{e}l Rouyer
\and Costin V\^{\i}lcu}
\title{Moderate smoothness of \\most Alexandrov surfaces}
\maketitle

\abstract{We show that, in the sense of Baire categories, a typical Alexandrov surface with curvature bounded below by $\kappa$ has no conical points. We use this result to prove that, on such a surface (unless it is flat), 
at a typical point, the lower and the upper Gaussian curvatures are equal to $\kappa$ and $\infty$ respectively.}

\bigskip

{\small Math. Subj. Classification (2010): 53C45}



\section{Introduction}

In this paper, an \textit{Alexandrov surface} will mean a compact
$2$-dimensional Alexandrov space with curvature bounded below, without
boundary. For the precise definition and basic properties, see \cite{BGP} or
\cite{Sh}. It is known that these surfaces are $2$-dimensional topological
manifolds. It is also known that, endowed with the Gromov-Hausdorff distance,
the set of all Alexandrov surfaces, together with their lower dimensional
limits, is a complete metric space (see the next section for details).

By the work of A. D. Alexandrov (for the existence) and A. V. Pogorelov (for
the uniqueness), any Alexandrov surface with curvature bounded below by $0$
and homeomorphic to the sphere $S^{2}$ can be realized as a unique (up to
rigid motions) convex surface (i.e., boundary of a compact convex set with
interior points) in $\mathbb{R}^{3}$. Therefore, the intrinsic geometry of
convex surfaces can be seen as a particular case of the geometry of Alexandrov
surfaces. Nevertheless, due to Pogorelov's rigidity theorem, the proofs for
intrinsic properties of convex surfaces generally involve extrinsic arguments.

\medskip

We recall here a few intrinsic properties of convex surfaces. The space
$\mathcal{S}$ of all convex surfaces, endowed with the usual Pompeiu-Hausdorff
metric, is a Baire space. In any Baire space, a property enjoyed by all
elements except those in a first category set is said to be \emph{typical}. We
also say that \emph{most} elements enjoy that property. For typical properties
of convex surfaces, we refer to \cite{gw} or \cite{z-b}. For Baire category
results in some variations of the space $\mathcal{S}$, see e.g. \cite{z-cc}
and \cite{v}.

The study of typical\textit{\ }properties of convex surfaces started with a
result of V. Klee \cite{k}, rediscovered and completed by P. Gruber
\cite{gru}: \textit{most convex surfaces are of differentiability class
}$C^{1}\setminus C^{2}$\textit{\ and strictly convex}; see also \cite{z-na}.
In particular, most convex surfaces have no conical points. Theorem \ref{sing}
is a generalization of this latter fact to Alexandrov surfaces, and is a key
ingredient in the proof of Theorem \ref{g}.

The description of most convex surfaces was successively improved by several
authors. For example, R. Schneider \cite{s}, T. Zamfirescu \cite{z-cp},
\cite{z-z}, \cite{z-pac}, K. Adiprasito \cite{Ad}, K. Adiprasito and T.
Zamfirescu \cite{AZ_1}, studied lower and upper \emph{directional curvatures}.
Consider a convex surface $S$ of differentiability class $C^{1}$ and a point
$x\in S$. The \emph{lower} and \emph{upper curvature }at $x$ in direction
$\tau$ are defined by $\gamma_{i}^{\tau}(x)=\liminf_{z\rightarrow x}\frac
{1}{r_{z}}$ and $\gamma_{s}^{\tau}(x)=\limsup_{z\rightarrow x}\frac{1}{r_{z}}$
respectively, where $r_{z}$ is the radius of the circle through $x$ and $z\in
S$ whose center belongs to the line normal to $S$ at point $x$, and $\tau$ is
the direction tangent to $S$ at point $x$ \textquotedblleft
toward\textquotedblright$z$. See any of the aforementioned papers for the
precise definition. T. Zamfirescu proved that, \textit{on most convex surfaces
}$S$\textit{, at each point }$x\in S$\textit{, }$\gamma_{i}^{\tau}%
(x)=0$\textit{ or }$\gamma_{s}^{\tau}(x)=\infty$\textit{, for any tangent
direction }$\tau$\textit{ at }$x$\textit{ }\cite{z-z}\textit{. Moreover, both
equalities hold simultaneously at most point }$x\in S$ \cite{z-cp}\textit{.
Still on most convex surfaces, }$\gamma_{s}^{\tau}(x)=0$\textit{ almost
everywhere, in any tangent direction }$\tau$\textit{ }\cite{z-z}. See
\cite{z-pac} for other results of the same flavour. One can also mention more
recent works on the existence of umbilical points of infinite curvature
\cite{Ad}, \cite{Sch}.

The above results inspired our investigations about Gaussian curvatures of
Alexandrov surfaces (Theorem \ref{g}), even though the proof techniques
involved are very different. Moreover, the notion of directional curvature is
essentially extrinsic, and admits no counterpart in the framework of
Alexandrov surfaces. Adapting the proof of Theorem \ref{g} for convex surfaces
yields a new result in this framework, Theorem \ref{Cor}. It should be noticed
that the relations between directional and Gaussian curvatures of convex
surfaces are hitherto not well understood, and it remains unclear whether
Theorem \ref{Cor} can be deduced from the previous results.

\medskip

T. Zamfirescu \cite{z-edp} discovered that \textit{on most convex surfaces,
most points are interior to no geodesic}, and his result was very recently
extended by K. Adiprasito and himself \cite{AZ_2} to Alexandrov surfaces, thus
showing that \textit{most Alexandrov surfaces are not Riemannian}. This seems
to be the first typical property established for Alexandrov surfaces.

In this paper we present the space of Alexandrov surfaces ($\S $2), and study
the existence of conical points on most Alexandrov surfaces ($\S $3); this
enables us to determine the lower and upper curvatures at most points on most
Alexandrov surfaces ($\S $4).

Another motivation for our paper comes from the work of Y. Machigashira on 
lower and upper curvatures \cite{m}; he proved that, for any Alexandrov surface, 
these curvatures are equal almost everywhere (see $\S $4 for definitions).
Our Theorem \ref{g} shows that, excepting flat surfaces, 
these curvatures are different at most points of most Alexandrov surfaces.

In another paper \cite{JC}, we study the existence of simple closed geodesics
on most Alexandrov surfaces: it depends on both the lower curvature bound and
the topology of surfaces.

\medskip

In any Alexandrov space $A$, a shortest path between two points is called a
\emph{segment}. The open (resp. closed) ball of radius $r$ centered at $x$
will be denoted by $B^{A}(x,r)$ (resp. $\bar{B}^{A}(x,r)$). When no confusion
is possible, the superscript $A$ will be omitted. Given a subset $C$ of $A$,
we denote by $\mu\left(  C\right)  $ its $2$-dimensional Hausdorff measure.
The length of a curve $c$ is denoted by $\ell\left(  c\right) $.

The length of the space of directions at a point $p\in A$ is called the
\emph{total angle} at $p$. The \emph{singular curvature} of $p$, denoted by
$\omega\left(  p\right)  $, is defined as $2\pi$ minus the total angle at $p$.
It is known that $\omega(p)\geq0$. A point with non-zero singular curvature is
said to be \emph{conical}.


\section{The space of Alexandrov surfaces}

The results presented in this section seem to be known, but not so easy to
find in the literature.

If $X$ and $Y$ are compact metric spaces, a correspondence between $X$ and $Y
$ is a relation $R$ such that for any $x\in X$ there is at least one $y\in Y$
satisfying $xRy$, and conversely, for any $y\in Y$ there is at least one $x\in
X$ such that $xRy$. The distortion $\mathrm{dis}\left(  R\right)  $ of $R$ is
defined by
\[
\mathrm{dis}\left(  R\right)  = \sup\left\{  \left\vert d\left(  x_{1},
x_{2}\right)  -d\left(  y_{1},y_{2}\right)  \right\vert | x_{1},x_{2}\in
X,~y_{1},y_{2}\in Y,~x_{1}Ry_{1},~x_{2}Ry_{2} \right\}
\]

One way to define the Gromov-Hausdorff distance $d_{GH}$ is to put
\[
d_{GH}\left(  X,Y\right)  =\frac{1}{2}\inf_{R}\mathrm{dis}\left(  R\right)
\text{,}%
\]
where the infimum is taken over all correspondences between $X$ and $Y$. It is
known that $d_{GH}$ is a metric on the set $\mathfrak{M}$ of all compact
metric spaces up to isometry, and that $\mathfrak{M}$ is complete with respect
to this distance \cite{P}. We will often use the following technical lemma.

\begin{lm}
\label{LSE} \cit{roTA} Let $\left\{  X_{n}\right\}  _{n\in\mathbb{N}}$ be a
sequence of elements of $\mathfrak{M}$ converging to $X$, and $\left\{
\varepsilon_{n}\right\}  _{n\in\mathbb{N}}$ a sequence of positive numbers.
Then there exist a compact metric space $Z$, an isometric embedding
$g:X\rightarrow Z$, and for each positive integer $n$, an isometric embedding
$f_{n}:X_{n}\rightarrow Z$, such that $d_{H}^{Z}\left(  f_{n}\left(
X_{n}\right)  ,g\left(  X\right)  \right)  <d_{GH}\left(  X_{n},X\right)
+\varepsilon_{n}$.
\end{lm}

The set $\mathcal{A}^{d}\left(  \kappa\right)  \subset\mathfrak{M}$ of
(isometry classes of) compact Alexandrov spaces of curvature at least $\kappa$
and dimension at most $d$ is known to be closed in $\mathfrak{M}$, and
therefore complete \cite[10.8.25]{bbi}. Hence the set $\mathcal{A}\left(
\kappa\right)  =\mathcal{A}^{2}\left(  \kappa\right)  \backslash
\mathcal{A}^{1}\left(  \kappa\right)  $ of all Alexandrov surfaces is open in
a Baire space, and consequently is itself a Baire space.

$\mathcal{A}\left(  \kappa\right)  $ obviously contains the set $\mathcal{R}%
\left(  \kappa\right)  $ of smooth compact Riemannian surfaces of Gaussian
curvature at least $\kappa$ everywhere.

Polyhedra are other examples of Alexandrov surfaces. Let $\M_{\kappa}$ denote
the simply connected $2$-dimensional manifold with constant curvature $\kappa$. 
A surface $P$ obtained by gluing a finite collection $\left\{T_{i}\right\}$ 
of geodesic triangles of $\M_{\kappa}$ will be called a $\kappa$-\emph{polyhedron}. 
Here, gluing means identifying parts of equal length of their boundaries $\partial T_{i}$, 
in such a way that the resulting topological space is a topological manifold. 
Denote by $\Gamma$ the set of those curves on $P$
which intersects the image of $\bigcup_{i}\partial T_{i}$ in finitely many
points. For $\gamma\in\Gamma$, we define $\ell_{0}\left(  \gamma\right)  $ as
the sum over $i$ of the lengths of the traces of $\gamma$ on each $T_{i}$.
Then, we define a metric on $P$: the distance between two points is the
infimum of $\ell_{0}\left(  \gamma\right)  $ over all those curves $\gamma
\in\Gamma$ joining them. Thus $P$ becomes a length space, and the length
induced by this metric coincides with $\ell_{0}$ for any $\gamma\in\Gamma$. A
point of a $\kappa$-polyhedron which is not the image of some vertex of some triangle
$T_{i}$ admits a neighborhood isometric to a ball of $\M_{\kappa}$. Hence, 
there are at most finitely many points which do not admit such a neighborhood; 
they are called the \emph{vertices} of $P$. For more details, see \cite{AZ}. 
The set of $\kappa$-polyhedra will be denoted by $\mathcal{P} \left( \kappa\right)$.

It's easy to see that a $\kappa$-polyhedron is an Alexandrov surface (with
curvature bounded below by any $\kappa^{\prime}\leq\kappa$) if and only if, at
each vertex $p$, image of the points $p_{1}\in\partial T_{i_{1}}$, \ldots\ ,
$p_{k}\in\partial T_{i_{k}}$, the sum over $j$ of the angles of $\partial
T_{i_{j}}$ at $p_{j}$ is at most $2\pi$.

It is known that Alexandrov surfaces are topological manifolds \cite{BGP}.
Moreover they are $2$-\emph{dimensional manifolds with bounded integral
curvature}, as defined in \cite{AZ}; this follows from the existence of a
curvature measure on Alexandrov surfaces \cite{m}. In \cite{AZ}, it is proved
that any $2$-dimensional manifold with bounded integral curvature can be
decomposed into a finite union of geodesic triangles whose interiors are
disjoint. The following lemma is a particular case of a stronger theorem in
\cite{AZ}.

\begin{lm}
\label{LCP} Let $M$ be a $2$-dimensional manifold with bounded integral
curvature. Let $\tau_{n}=\left\{  T_{n}^{i}\right\}  _{i=1}^{q_{n}}$ be a
decomposition of $M$ into geodesic triangles with disjoint interiors, such
that $\max_{1\leq i\leq q_{n}}\mathrm{diam}\left(  T_{n}^{i}\right)
\underset{n\rightarrow\infty}{\longrightarrow}0$. Denote by $P_{n}$ the
$0$-polyhedron obtained by replacing each triangle $T_{n}^{i}$ by a Euclidean
one with the same side lengths. Then $P_{n}$ converges to $M$ with respect to
the Gromov-Hausdorff convergence.
\end{lm}

\begin{lm}
\label{LDP} $\mathcal{P}\left(  \kappa\right)  \cap\mathcal{A}\left(
\kappa\right)  $ is dense in $\mathcal{A}\left(  \kappa\right)  $.
\end{lm}

\begin{proof}
Let $S\in\mathcal{A}\left(  \kappa\right)  $. Let $\left\{  P_{n}^{0}\right\}
$ be a sequence of $0$-polyhedra constructed by gluing Euclidean triangles
corresponding to finer and finer triangulations of $S$, and let$\ P_{n}%
^{\kappa}$ be the $\kappa$-polyhedron obtained from $P_{n}^{0}$ by replacing
each Euclidean triangle by a triangle of $\M_{\kappa}$. Since $S\in
\mathcal{A}\left(  \kappa\right)  $, each angle of a face of $P_{n}^{\kappa}$
is not larger than the corresponding angle in $S$. It follows that the sum of
angles glued at each vertex of $P_{n}^{\kappa}$ is less than or equal to the
corresponding sum in $S$, which in turn is less than or equal to $2\pi$. Hence
$P_{n}^{\kappa}\in\mathcal{A}\left(  \kappa\right)  $. By Lemma \ref{LCP},
$P_{n}^{0}$ converges to $S$, so it is sufficient to prove that $d_{GH}\left(
P_{n}^{0},P_{n}^{\kappa}\right)  $ tends to $0$. Let $\Theta^{\kappa}\left(
A,B,C,s,t\right)  $ be the distance on $\M_{\kappa}$ between points $p$ and
$q$, where $a$, $b$, $c$, $p$, $q\in\M_{\kappa}$ are such that (see Figure
\ref{F1})%
\[%
\begin{array}
[t]{ccc}%
d\left(  a,b\right)  =B\text{,} & d\left(  a,p\right)  =sB\text{,} & d\left(
p,b\right)  =\left(  1-s\right)  B\text{,}\\
d\left(  a,c\right)  =C\text{,} & d\left(  a,q\right)  =tC\text{,} & d\left(
q,c\right)  =\left(  1-t\right)  C\text{,}\\
d\left(  b,c\right)  =A\text{.} &  &
\end{array}
\]
It is well-known, that
\[
\lim_{\delta\rightarrow0}\frac{\Theta^{\kappa}\left(  \delta A,\delta B,\delta
C,s,t\right)  }{\Theta^{0}\left(  \delta A,\delta B,\delta C,s,t\right)
}=1\text{,}%
\]
and the convergence is uniform with respect to the variables $A$, $B$, $C$,
$s$, $t$.%

\begin{figure}
[tbh]
\begin{center}
\includegraphics[
height=1.11in,
width=2.0415in
]%
{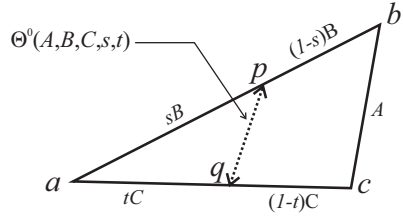}%
\caption{Definition of $\Theta^{\kappa}$.}%
\label{F1}%
\end{center}
\end{figure}

Let $R_{n}$ be the correspondence between $P_{n}^{0}$ and $P_{n}^{\kappa}$
defined as follows: $xR_{n}y$ if and only if $x$ and $y$ belong to
corresponding triangles.

Let $\delta_{n}$ be the maximum of the diameters of the faces of $P_{n}^{0}$.
Choose $\varepsilon>0$. For $n$ large enough,
\[
\frac{\Theta^{0}\left(  \delta_{n}A,\delta_{n}B,\delta_{n}C,s,t\right)
}{\Theta^{\kappa}\left(  \delta_{n}A,\delta_{n}B,\delta_{n}C,s,t\right)
}>1-\varepsilon\text{.}%
\]
Take $x$, $y\in P_{n}^{0}$, $x^{\prime}$, $y^{\prime}\in P_{n}^{\kappa}$ such
that $xRx^{\prime}$ and $yRy^{\prime}$. Let $\sigma_{n}$ be a segment on
$P_{n}^{0}$ between $x$ and $y$; this segment crosses the edges (i.e., sides
of triangles) of the decomposition at the points $x_{n}^{1}$, \ldots,
$x_{n}^{k}$. These points have natural counterparts $x_{n}^{\prime1}$, \ldots,
$x_{n}^{\prime k}$ in $P_{n}^{\kappa}$ (on the same edge, at the same
distances from its endpoints). Let $\sigma_{n}^{\prime}\subset P_{n}^{\kappa}$
be a simple path composed of the segments from $x_{n}^{^{\prime}i}$ to
$x_{n}^{^{\prime}i+1}$ ($i=1,...,k-1$). For $n$ large enough, we have
\begin{align*}
d_{P_{n}^{0}}\left(  x,y\right)  =\ell\left(  \sigma\right)   &  \geq\left(
1-\varepsilon\right)  \ell\left(  \sigma^{\prime}\right)  -2\delta_{n}%
\geq\left(  1-\varepsilon\right)  d_{P_{n}^{\kappa}}\left(  x_{n}^{\prime
1},x_{n}^{\prime k}\right)  -2\delta_{n}\\
&  \geq\left(  1-\varepsilon\right)  d_{P_{n}^{\kappa}}\left(  x^{\prime
},y^{\prime}\right)  -4\delta_{n}\geq\left(  1-2\varepsilon\right)
d_{P_{n}^{\kappa}}\left(  x^{\prime},y^{\prime}\right)  \text{.}%
\end{align*}
An opposite inequality can be obtained in exactly the same way. Hence,
$\mathrm{dis}\left(  R_{n}\right)  $ tends to $0$.
\end{proof}

\begin{lm}
\label{LDR} $\mathcal{R}\left(  \kappa\right)  $ is dense in $\mathcal{A}%
\left(  \kappa\right)  $.
\end{lm}

\begin{proof}
By Lemma \ref{LDP}, it is sufficient to approximate a ball $B=B\left(
o,R\right)  $ of a $\kappa$-polyhedron, centered at a vertex $o$ of positive
curvature $\omega$, by a Riemannian ball which has constant curvature $\kappa$
near its boundary. By applying a homothety, we can assume that $\kappa=0$ or
$\pm1$. We treat only the case $\kappa=1$, the reader will easily adapt the
proof for cases $\kappa=0$, $-1$.

Note that $B\setminus\{o\}$ is isometric to $]0,R[\times\mathbb{R}%
/2\pi\mathbb{Z}$ equipped with the metric
\begin{equation}
ds^{2}=dr^{2}+\left(  1-\frac{\omega}{2\pi}\right)  ^{2}\sin\left(  r\right)
^{2} d\theta^{2}\text{.} \label{2}%
\end{equation}

For any $\lambda,\varepsilon>0$, we define $k_{\lambda,\varepsilon}%
:\mathbb{R}\rightarrow\mathbb{R}$ by
\[
k_{\lambda,\varepsilon}\left(  t\right)  =\left\{
\begin{array}
[c]{l}%
1\\
\frac{\lambda^{2}}{\varepsilon^{2}}%
\end{array}
\right.
\begin{array}
[c]{l}%
\text{if }t\notin\left[  \varepsilon,2\varepsilon\right]  ,\\
\text{if }t\in\left[  \varepsilon,2\varepsilon\right]  .
\end{array}
\]
There is a unique $C^{1}$ ($C^{\infty}$ on $\mathbb{R}\setminus\left\{
\varepsilon,2\varepsilon\right\}  $) function $f_{\lambda,\varepsilon
}:\mathbb{R}\rightarrow\mathbb{R}$ such that $f\left(  0\right)  =0$,
$f^{\prime}\left(  0\right)  =1$, and, for any $t\neq\varepsilon$,
$2\varepsilon$,
\begin{equation}
f_{\lambda,\varepsilon}^{\prime\prime}\left(  t\right)  +k_{\lambda
,\varepsilon}\left(  t\right)  f_{\lambda,\varepsilon}=0\text{.} \label{1}%
\end{equation}
Explicitly, for $t>2\varepsilon$,
\[
f_{\lambda,\varepsilon}\left(  t\right)  =A_{\lambda,\varepsilon}\sin
t+B_{\lambda,\varepsilon}\cos t,
\]
where
\begin{align*}
A_{\lambda,\varepsilon}  &  =\frac{1}{2\varepsilon\lambda}\left(
\begin{array}
[c]{l}%
3\left(  \varepsilon^{2}-\lambda^{2}\right)  \sin\varepsilon\sin\lambda
\cos^{2}\varepsilon+2\varepsilon\lambda\cos\lambda\cos\varepsilon\\
\hspace{1.45cm}+\sin\varepsilon\left(  \varepsilon^{2}+\lambda^{2}+\left(
\lambda^{2}-\varepsilon^{2}\right)  \sin^{2}\varepsilon\right)  \sin\lambda
\end{array}
\right) \\
B_{\lambda,\varepsilon}  &  =\frac{1}{\varepsilon\lambda}\left(
\cos\varepsilon\left(  \lambda^{2}+\left(  \varepsilon^{2}-\lambda^{2}\right)
\cos2\varepsilon\right)  \sin\lambda-\varepsilon\lambda\cos\lambda
\sin\varepsilon\right)  \text{.}%
\end{align*}
Note that
\begin{align*}
\lim_{\varepsilon\rightarrow0}A_{\lambda,\varepsilon}  &  =\cos\lambda
-\lambda\sin\lambda\text{,}\\
\lim_{\varepsilon\rightarrow0}B_{\lambda,\varepsilon}  &  =0\text{.}%
\end{align*}
If we replace in (\ref{1}) $k_{\lambda,\varepsilon}$ by a smooth function
$k_{\lambda,\varepsilon,\eta}$ which equals $k_{\lambda,\varepsilon}$ outside
$\left[  \varepsilon-\eta,\varepsilon+\eta\right]  \cup\left[  2\varepsilon
-\eta,2\varepsilon+\eta\right]  $, then the corresponding solution of
(\ref{1}) tends to $f_{\lambda,\varepsilon}$ when $\eta$ tends to $0$. It
follows that, for arbitrarily small $\tau>0$, one can find a smooth function
$f$ such that $f=\sin$ on a small enough neighborhood of $0$, $-f^{\prime
\prime}/f\geq1$, and for $t>\tau$, $f\left(  t\right)  =\left(  1-\frac
{\omega}{2\pi}\right)  \sin\left(  t+\phi\right)  $, with $\left\vert
\phi\right\vert <\tau$. Now, the metric $dr^{2}+f\left(  r\right)  ^{2}%
d\theta^{2}$ is a suitable approximation of (\ref{2}).
\end{proof}


\section{Conical points}

The goal of this section is to prove the following result.

\begin{thm}
\label{sing} Most Alexandrov surfaces in $\mathcal{A}\left(  \kappa\right)  $
have no conical points.
\end{thm}

Before proving this theorem, we establish a lemma.

\begin{lm}
\label{LCrv}Let $Z$ be a compact metric space, and $A_{n}$ a sequence of
Alexandrov surfaces, isometrically embedded in $Z$, converging to $A\subset Z$
for the Hausdorff distance of $Z$. Let $p_{n}$ be a point on $A_{n}$ and
assume that $p_{n}$ converges to $p\in A$.

i) If $\sigma_{n}$,$\gamma_{n}$ are segments of $A_{n}$ emanating from $p_{n}$
and converging to segments $\sigma$, $\gamma\subset A$, then $\liminf
\measuredangle\left(  \sigma_{n},\gamma_{n}\right)  \geq\measuredangle\left(
\sigma,\gamma\right)  $.

ii) $\omega\left(  p\right)  \geq\lim\sup\omega\left(  p_{n}\right)  $.

iii) If $\omega\left(  p\right)  =0$ then $\lim\measuredangle\left(
\sigma_{n},\gamma_{n}\right)  =\measuredangle\left(  \sigma,\gamma\right)  $.
\end{lm}

\begin{proof}
Following \cite{BGP}, we denote by $\tilde{\measuredangle}abc$ the angle at
point $\tilde{b}$ of a geodesic triangle $\tilde{a}\tilde{b}\tilde{c}%
\subset\M_{\kappa}$ such that $d(\tilde{a},\tilde{b})=d\left(  a,b\right)  $,
$d(\tilde{b},\tilde{c})=d\left(  b,c\right)  $ and $d\left(  \tilde{a}%
,\tilde{c}\right)  =d\left(  a,c\right)  $. We recall that the angle
$\measuredangle\left(  \sigma,\gamma\right)  $ between two segments $\sigma$,
$\gamma$ emanating for $p$ is by definition the limit of $\tilde
{\measuredangle}\sigma\left(  t\right)  p\gamma\left(  s\right)  $ when $s$
and $t$ both tend to $0$.

Assume that the result fails, that is, there exists $\varepsilon>0$ such that
for arbitrarily large $n$, we have%
\[
\measuredangle\left(  \sigma,\gamma\right)  >\measuredangle\left(  \sigma
_{n},\gamma_{n}\right)  +\varepsilon\text{.}%
\]
Let $\tau>0$ be small enough to ensure that%
\[
\measuredangle\left(  \sigma,\gamma\right)  \leq\tilde{\measuredangle}%
\sigma\left(  \tau\right)  p\gamma\left(  \tau\right)  +\frac{\varepsilon}%
{3}\text{.}%
\]
Let $q_{n}$, $r_{n}\in A_{n}$ be such that $\max\left(  d\left(  q_{n}%
,\sigma\left(  \tau\right)  \right)  ,d\left(  r_{n},\gamma\left(
\tau\right)  \right)  \right)  \leq d_{H}^{Z}\left(  A,A_{n}\right)  $. For
$n$ large enough, we have%
\[
\left\vert \tilde{\measuredangle}\sigma\left(  \tau\right)  p\gamma\left(
\tau\right)  -\tilde{\measuredangle}q_{n}pq_{n}\right\vert <\frac{\varepsilon
}{3}\text{.}%
\]
Now,%
\begin{align*}
\measuredangle\left(  \sigma_{n},\gamma_{n}\right)   &  \geq\tilde
{\measuredangle}q_{n}pr_{n}\geq\tilde{\measuredangle}\sigma\left(
\tau\right)  p\gamma\left(  \tau\right)  -\frac{\varepsilon}{3}\\
&  \geq\measuredangle\left(  \sigma,\gamma\right)  -\frac{2\varepsilon}{3}%
\geq\measuredangle\left(  \sigma_{n},\gamma_{n}\right)  +\frac{\varepsilon}%
{3}\text{,}%
\end{align*}
and a contradiction is obtained, thus proving the first statement.

Consider a segment $\zeta$ emanating from $p$, and such that $2\pi
-\omega\left(  p\right)  =\measuredangle\left(  \sigma,\gamma\right)
+\measuredangle\left(  \gamma,\zeta\right)  +\measuredangle\left(
\zeta,\sigma\right)  $. Choose $s$, $g$, $z$, some interior points of $\sigma
$, $\gamma$, and $\zeta$ respectively. Take $s_{n}$, $g_{n}$, $z_{n}\in A_{n}$
converging to $s$, $g$, $z$ respectively. Let $\sigma_{n}$ (resp. $\gamma_{n}%
$, $\zeta_{n}$) be a segment between $p_{n}$ and $s_{n}$ (resp. $g_{n}$,
$z_{n}$). Then $\sigma_{n}$ (resp. $\gamma_{n}$, $\zeta_{n}$) converges to the
part of $\sigma$ (resp. $\gamma$, $\zeta$) between $p$ and $s$ (resp. $g$,
$z$), for there is only one segment between these points. By \textit{(i)},
\begin{align*}
\limsup\omega\left(  p_{n}\right)   &  =2\pi-\liminf\left(  \measuredangle
\left(  \sigma_{n},\gamma_{n}\right)  +\measuredangle\left(  \gamma_{n}%
,\zeta_{n}\right)  +\measuredangle\left(  \zeta_{n},\gamma_{n}\right)  \right)
\\
&  \leq2\pi-\left(  \measuredangle\left(  \sigma,\gamma\right)
+\measuredangle\left(  \gamma,\zeta\right)  +\measuredangle\left(
\zeta,\sigma\right)  \right)  =\omega\left(  p\right)  \text{.}%
\end{align*}
This proves \textit{(ii)}.

Now, if $\omega\left(  p\right)  =0$, the above inequality must be an
equality, implying \textit{(iii)}.
\end{proof}

\begin{proof}
[Proof of Theorem \ref{sing}]Let $\mathcal{M}\left(  a\right)  $, $a>0$, be
the set of Alexandrov surfaces with curvature bounded below by $\kappa$ which
admits a point of singular curvature greater than or equal to $a$.

We claim that $\mathcal{M}\left(  a\right)  $ is closed. Take a sequence
$\left\{  A_{n}\right\}  $ of surfaces of $\mathcal{M}\left(  a\right)  $
converging to $A$ in $\mathcal{A}\left(  \kappa\right)  $. By Lemma \ref{LSE},
we can assume that each $A_{n}$ and $A$ are all isometrically embedded in some
compact metric space $Z$, and that $A_{n}$ converges to $A$ with respect to
the Hausdorff distance of $Z$. Let $p_{n} \in A_{n}$ be a point of curvature
at least $a$. Let $p\in A$ be a limit point of the sequence $\left\{
p_{n}\right\}  $; by Lemma \ref{LCrv}, $\omega\left(  p\right)  \geq a$ and
$A\in\mathcal{M}\left(  a\right)  $.

Moreover, due to the density of $\mathcal{R}\left(  \kappa\right)  $ in
$\mathcal{A}\left(  \kappa\right)  $, $\mathcal{M}\left(  a\right)  $ has
empty interior. Hence
\[
\bigcup_{n\in\mathbb{N}}\mathcal{M}\left(  \frac{1}{n+1}\right)  =\left\{
A\in\mathcal{A}\left(  \kappa\right)  |\exists p\in A,~\omega\left(  p\right)
>0\right\}
\]
is meager, and the conclusion follows.
\end{proof}


\section{Lower and upper curvatures}

\label{SGC}For any Alexandrov space $A$ with curvature bounded below, and any
geodesic triangle $\Delta$ in $A$, denote by $\sigma_{0}(\Delta)$ the area of
the Euclidean triangle with sides of the same length as $\Delta$, and by
$e_{0}(\Delta)$ the excess of $\Delta$ (the sum of its angles minus $\pi$).
Let us denote by $E\left(  x,\delta,a\right)  $ the set of all triangles of
diameter less than $\delta$, such that $x$ is interior to $\Delta$, and such
that each angle of $\Delta$ is at least $a$.

The \textit{lower} and the \textit{upper curvature at }$x$, $\underline{K}(x)$
and $\overline{K}(x)$, were defined by A. D. Alexandrov (see the survey
\cite{abn}) by
\[
\underline{K}(x)=\lim_{\delta\rightarrow0}\inf_{\Delta\in E\left(
x,\delta,0\right)  }\frac{e_{0}(\Delta)}{\sigma_{0}(\Delta)}\text{,}%
\hspace{2.0454pc}\overline{K}\left(  x\right)  =\lim_{\delta\rightarrow0}%
\sup_{\Delta\in E\left(  x,\delta,0\right)  }\frac{e_{0}(\Delta)}{\sigma
_{0}(\Delta)}\text{.}%
\]
Y. Machigashira defined the lower and the upper curvature at $x$,
$\underline{G}(x)$ and $\overline{G}\left(  x\right)  $, with slightly more
sophisticated formulas \cite{m}:
\[
\underline{G}(x)=\lim_{a\rightarrow0}\underline{G}_{a}(x)\text{,}%
\hspace{2.0302pc}\ \overline{G}\left(  x\right)  =\lim_{a\rightarrow
0}\overline{G}_{a}\left(  x\right)  \text{,}%
\]
where
\[
\underline{G}_{a}(x)=\lim_{\delta\rightarrow0}\inf_{\Delta\in E\left(
x,\delta,a\right)  }\frac{e_{0}(\Delta)}{\mu(\Delta)}\text{,}\hspace
{2.0302pc}\overline{G}_{a}\left(  x\right)  =\lim_{\delta\rightarrow0}%
\sup_{\Delta\in E\left(  x,\delta,a\right)  }\frac{e_{0}(\Delta)}{\mu(\Delta
)}\text{.}%
\]
This new definition allowed him to prove that, for any Alexandrov surface,
$\underline{G}(x)=\overline{G}\left(  x\right)  $ almost everywhere; therefore
a variant of the Gauss-Bonnet theorem holds.

Let $\mathcal{A}^{+}\left(  0\right)  $ be the space of Alexandrov surfaces
with curvature bounded below by $0$ which have a positive Euler
characteristic. The aim of this section is to prove the following result.

\begin{thm}
\label{g} For most surfaces $A\in\mathcal{A}\left(  \kappa\right)  $, at most
points $x\in A$, $\underline{K}(x)=\underline{G}(x)=\kappa$.

If $\kappa\neq0$, then for most surfaces $A\in\mathcal{A}\left(
\kappa\right)  $, at most points $x\in A$, $\overline{K}(x)=\overline
{G}(x)=\infty$. Otherwise, for most surfaces $A\in\mathcal{A}^{+}\left(
0\right)  $, at most points $x\in A$, $\overline{K}(x)=\overline{G}(x)=\infty$.
\end{thm}

\begin{remark}
The restriction of the second statement is necessary if $\kappa=0$. Indeed,
one can easily prove that a topological torus (or a Klein bottle) in
$\mathcal{A}\left(  0\right)  $ is necessarily flat. Since by Perel'man's
stability theorem (see for instance \cite[10.10.5]{bbi}) the tori form an open
set in $\mathcal{A}\left(  0\right)  $, it is not possible to expect any kind
of roughness for most surfaces in $\mathcal{A}\left(  0\right)  $. This fact
was already pointed out by K. Adiprasito and T. Zamfirescu \cite{AZ_2}; see
also \cite{JC0}.
\end{remark}

We need the following lemma for the proof of Theorem \ref{g}.

\begin{lm}
\label{42} Let $A\in\mathcal{A}\left(  \kappa\right)  $, $x\in A$ be a
non-conical point, and $a$ be a (small) fixed positive number. Let $\left\{
\Delta_{n}\right\}  $ be a sequence of geodesic triangles of $E\left(
x,1,a\right)  $ converging to $x$. Then $\frac{\sigma_{0}\left(  \Delta
_{n}\right)  }{\mu\left(  \Delta_{n}\right)  }$ converges to $1$.
\end{lm}

\begin{proof}
This follows from the fact that $\Delta_{n}$ is $\varepsilon_{n}$-isometric
(see \cite{BGP}) to a Euclidean triangle for a sequence $\left\{
\varepsilon_{n}\right\}  $ tending to $0$. For details, see \cite[Lemma
2.4]{m}.
\end{proof}

As a consequence of Lemma \ref{42}, one can substitute $\sigma_{0}$ to $\mu$
in the definition of $\underline{G}_{a}(x)$ and $\overline{G}_{a}$ (both
definitions give $\infty$ if $x$ is conical). It follows that for all $x\in A$
and any fixed (small) number $a>0$, we have
\[
\underline{K}(x)\leq\underline{G}(x)\leq\underline{G}_{a}(x)\leq\overline
{G}_{a}\left(  x\right)  \leq\overline{G}(x)\leq\overline{K}\left(  x\right)
\text{.}%
\]

A small geodesic triangle $\Delta$ has a well defined interior, and a well
defined exterior. Among all triangles sharing the same vertices, one has the
smallest (with respect to inclusion) interior; we denote it by $\underline
{\Delta}$. Similarly, the triangle with the same vertices as $\Delta$ and the
largest interior will be denoted by $\overline{\Delta}$. Let $F\left(
x,\delta,a\right)  $ be the set of those triangles $\Delta$ such that
$x\in\mathrm{int}\underline{\Delta}$, $\mathrm{\mathrm{diam}}\Delta<\delta$,
and each angle of $\underline{\Delta}$ is greater than $a$. We set
\begin{align*}
\underline{H}_{a}\left(  x\right)   &  =\lim_{\delta\rightarrow0}\underline
{H}_{a,\delta}\left(  x\right)  ,\hspace{2.0608pc}\underline{H}_{a,\delta
}\left(  x\right)  =\inf_{\Delta\in F\left(  x,\delta,a\right)  }\frac
{e_{0}(\overline{\Delta})}{\mathcal{\sigma}_{0}(\Delta)}\text{,}\\
\overline{H}_{a}\left(  x\right)   &  =\lim_{\delta\rightarrow0}\overline
{H}_{a,\delta}\left(  x\right)  ,\hspace{2.0685pc}\overline{H}_{a,\delta
}\left(  x\right)  =\sup_{\Delta\in F\left(  x,\delta,a\right)  }\frac
{e_{0}(\underline{\Delta})}{\mathcal{\sigma}_{0}(\Delta)}\text{.}%
\end{align*}
Since $e_{0}\left(  \overline{\Delta}\right)  \geq e_{0}\left(  \Delta\right)
\geq e_{0}\left(  \underline{\Delta}\right)  $ and $F\left(  x,\delta
,a\right)  \subset E\left(  x,\delta,a\right)  $, it is clear that
$\underline{G}_{a}(x)\leq\underline{H}_{a}\left(  x\right)  $ and
$\overline{H}_{a}\left(  x\right)  \leq\overline{G}_{a}\left(  x\right)  $ for
all points $x$ in $A$. Moreover, on surfaces without conical points,
$\underline{H}_{a,\delta}$ and $\overline{H}_{a,\delta}$ are respectively
upper and lower semi-continuous in the following strong sense.

\begin{lm}
\label{LSCH} Let $Z$ be a compact metric space. Let $\left\{  A_{n}\right\}  $
be a sequence of Alexandrov surfaces with curvature bounded below by $\kappa$,
embedded in $Z$ and converging to $A\in\mathcal{A}\left(  \kappa\right)$
with respect to the Hausdorff distance in $Z$. Assume that $A$ has no conical
points and let $x_{n}\in A_{n}$ converges to $x\in A$. Then
\[
\limsup_{n\rightarrow\infty}\underline{H}_{a,\delta}\left(  x_{n}\right)
\leq\underline{H}_{a,\delta}\left(  x\right)  , \hspace{0.3cm} \liminf
_{n\rightarrow\infty}\overline{H}_{a,\delta}\left(  x_{n}\right)
\geq\overline{H}_{a,\delta}\left(  x\right)  \text{.}%
\]

\end{lm}

\begin{proof}
By Lemma \ref{LCrv}, in the case of surfaces without conical points, the angle
between two segments of $A_{n}$ tends to the angle between the limit segments.

Fix $\varepsilon>0$; there exists a triangle $\Delta$ of $F\left(
x,\delta,a\right)  $ such that
\begin{equation}
\frac{e_{0}(\overline{\Delta})}{\mathcal{\sigma}_{0}(\Delta)}\leq\underline
{H}_{a,\delta}\left(  x\right)  +\frac{\varepsilon}{2}\text{.} \label{3}%
\end{equation}
Let $u$, $v$, $w$ be the vertices of $\Delta$. Obviously, one can choose
$\Delta$ such that $\Delta=\underline{\Delta}$. Let $u_{n}$, $v_{n}$, $w_{n}$
be points of $A_{n}$ converging to $u$, $v$, and $w$ respectively. Let
$\Delta_{n}=\overline{\Delta_{n}}$ be the fattest geodesic triangle with
vertices $u_{n}$, $v_{n}$, $w_{n}$. Let $\Delta^{\prime}$ (resp.
$\Delta^{\prime\prime}$) be a limit triangle of the sequence $\left\{
\underline{\Delta_{n}}\right\}  $ (resp. $\Delta_{n}$).

Since the diameter is continuous with respect to the Hausdorff distance,
$\mathrm{\mathrm{diam}}\Delta_{n}<\delta$ for $n$ large enough. Due to the
choice of $\Delta$, the angles of $\Delta^{\prime}$ are greater than the
corresponding angles of $\Delta$, hence, by continuity of angles, for $n$
large enough, the angles of $\underline{\Delta_{n}}$ are also greater than $a
$. The point $x$ belongs to $\mathrm{int}\Delta\subset\mathrm{int}%
\Delta^{\prime}$, hence, for $n$ large enough, $x_{n}\in\underline{\Delta_{n}%
}$. It follows that $\Delta_{n}\in F\left(  x_{n},\delta,a\right)  $, and
therefore
\begin{equation}
\underline{H}_{a,\delta}\left(  x_{n}\right)  \leq\frac{e_{0}(\Delta_{n}%
)}{\mathcal{\sigma}_{0}(\Delta_{n})}\text{.} \label{4}%
\end{equation}
For $n$ large enough, the continuity of angles implies
\begin{equation}
\frac{e_{0}(\Delta_{n})}{\mathcal{\sigma}_{0}(\Delta_{n})}\leq\frac
{e_{0}(\Delta^{^{\prime\prime}})}{\mathcal{\sigma}_{0}(\Delta)}+\frac
{\varepsilon}{2}\leq\frac{e_{0}(\overline{\Delta})}{\mathcal{\sigma}%
_{0}(\Delta)}+\frac{\varepsilon}{2}\text{.} \label{5}%
\end{equation}
Gathering (\ref{4}), (\ref{5}) and (\ref{3}), we get for $n$ large enough%
\[
\underline{H}_{a,\delta}\left(  x_{n}\right)  \leq\underline{H}_{a,\delta
}\left(  x\right)  +\varepsilon\text{.}%
\]
This holds for any $\varepsilon>0$, whence the conclusion concerning
$\underline{H}_{a,\delta}$.

The case of $\overline{H}_{a,\delta}$ is similar.
\end{proof}

\begin{lm}
\label{LDCP} Put $\mathcal{B}=\mathcal{A}\left(  \kappa\right)  $ if
$\kappa\neq0$, and $\mathcal{B}=\mathcal{A}^{+}\left(  0\right)  $ otherwise.
For any $A\in\mathcal{B}$ and any $\varepsilon>0$, there exists a $\kappa
$-polyhedron $P\in\mathcal{B}$ such that $d_{GH}\left(  A,P\right)
\leq\varepsilon$ and the conical points of $P$ form an $\varepsilon$-net in
$P$.
\end{lm}

\begin{proof}
Assume first that $\kappa>0$. Applying to $A\in\mathcal{A}\left(
\kappa\right)  $ a global homothety (of scaling factor slightly less than $1$)
yields a surface $A^{\prime}\in A\left(  \kappa^{\prime}\right)  $ with
$\kappa^{\prime}>\kappa$, arbitrarily close to $A$. By angle comparison, the
vertices of any $\kappa$-polyhedral approximation $P$ of $A^{\prime}$ (see
Lemma \ref{LDP}) are conical points.

The case $\kappa^{\prime}<0$ is similar (with a scaling factor more than $1 $).

Now consider $A\in\mathcal{A}^{+}\left(  0\right)  $. Suppose first that $A$
is homeomorphic to the sphere. A polyhedral approximation $P$ of $A$ is also
homeomorphic to the sphere; by Alexandrov's existence Theorem, it can be
realized as a convex polyhedron in $\mathbb{R}^{3}$. If the vertices of $P$
are too far from each other, one can add new ones in the following way. Choose
a point $p$ near $P$, outsides of $P$, close to the point you want to
\textquotedblleft make\textquotedblright\ conical, and take the convex hull of
$P\cup\left\{  p\right\}  $. If $A$ is homeomorphic to the projective plane,
one can do the same construction with the universal covering of $P$, by adding
pair of points $p$, $p^{\prime}$ symmetric to each other.
\end{proof}

Now we are in a position to prove Theorem \ref{g}.

\begin{proof}
[Proof of Theorem \ref{g}]Denote by $\mathcal{A}^{0}\left(  \kappa\right)  $
the set of Alexandrov surfaces with curvature bounded below by $\kappa$
without conical points. Observe that $\mathcal{A}^{0}\left(  \kappa\right)  $
is residual in $\mathcal{A}\left(  \kappa\right)  $ by Theorem \ref{sing}, and
so is itself a Baire space. For $A\in\mathcal{A}^{0}\left(  \kappa\right)  $,
we define
\[
C_{a,\delta,\varepsilon}^{A}=\set(:x\in A|\underline{H}_{a,\delta}\left(
x\right)  \geq\kappa+\varepsilon:)
\]
and
\begin{align*}
\mathcal{M}  &  =\set(:A\in\mathcal{A}^{0}\left(  \kappa\right)
|{\set[2](:x\in A|\underline{H}_{a}\left(  x\right)  >\kappa:)}~\text{is not
meager}:)\\
&  =\bigcup_{p\in\mathbb{N}^{\ast}}\set(:A\in\mathcal{A}^{0}\left(
\kappa\right)  |{\set[2](:x\in A|\underline{H}_{a,\frac{1}{p}}\left(
x\right)  >\kappa:)}~\text{is not meager}:)\\
&  =\bigcup_{p\in\mathbb{N}^{\ast}}\bigcup_{q\in\mathbb{N}^{\ast}}%
\set(:A\in\mathcal{A}^{0}\left(  \kappa\right)  |C_{a,\frac{1}{p},\frac{1}{q}%
}^{A}~\text{is not meager}:)\text{.}%
\end{align*}
Applying Lemma \ref{LSCH} (with $A_{n}=A$), we get that the sets
$C_{a,\delta,\varepsilon}^{A}$ are closed, and consequently, they are meager
if and only if they have empty interior. It follows that
\[
\mathcal{M}=\bigcup_{p\in\mathbb{N}^{\ast}}\bigcup_{q\in\mathbb{N}^{\ast}%
}\bigcup_{r\in\mathbb{N}^{\ast}}\mathcal{M}_{p,q,r},
\]
where
\[
\mathcal{M}_{p,q,r}\overset{\mathrm{def}}{=}\set(:A\in\mathcal{A}^{0}\left(
\kappa\right)  |\exists x\in A~\text{s.t.}~\bar{B}^{A}\left(  x,\frac{1}%
{r}\right)  \subset C_{a,\frac{1}{p},\frac{1}{q}}^{A}:)\text{.}%
\]
We claim that $\mathcal{M}_{p,q,r}$ is closed in $A^{0}\left(  \kappa\right)
$. Let $\left\{  A_{n}\right\}  $ be a sequence of surfaces of $\mathcal{M}%
_{p,q,r}$, converging to $A\in\mathcal{A}^{0}\left(  \kappa\right)  $. For
each $n$, there exists $x_{n}\in A_{n}$ such that $B_{n}\overset{\mathrm{def}%
}{=}\bar{B}^{A_{n}}\left(  x,\frac{1}{r}\right)  \subset C_{a,1/p,1/q}^{A_{n}%
}$. By selecting a subsequence, we may assume that $B_{n}$ converges to a ball
$B=\bar{B}^{A}\left(  x,1/r\right)  $. Now, a point $x\in B$ is limit of
points of $B_{n}$ and, by Lemma \ref{LSCH}, $\underline{H}_{a,1/p}\left(
x\right)  \geq\kappa+1/q$, whence $B\subset C_{a,1/p,1/q}^{A}$ and
$A\in\mathcal{M}_{p,q,r}$. This proves the claim.

Arbitrarily close to any surface of $\mathcal{M}_{p,q,r}$ is a $\kappa
$-polyhedron, and arbitrarily close to this polyhedron, is a Riemannian
surface whose Gaussian curvature is $\kappa$ outside a finite number of
arbitrarily small closed balls (see Lemma \ref{LDR}). Consequently
$\mathcal{M}_{p,q,r}$ has no interior points. Hence $\mathcal{M}$ is meager in
$\mathcal{A}^{0}\left(  \kappa\right)  $, and therefore in $\mathcal{A}\left(
\kappa\right)  $, too.

Put $\mathcal{B}^{0}=\mathcal{B}\cap\mathcal{A}^{0}$, where $\mathcal{B}$ is
the set defined in Lemma \ref{LDCP}. For $A\in\mathcal{B}^{0}$, we define%
\[
D_{a,\delta,q}^{A}=\set(:x\in A|\overline{H}_{a,\delta}\left(  x\right)  \leq
q:)
\]
and
\begin{align*}
\mathcal{N}  &  =\set(:A\in\mathcal{B}^{0}|{\set[2](:x\in A|\overline{H}%
_{a}\left(  x\right)  <\infty:)}~\text{is not meager}:)\\
&  =\bigcup_{p\in\mathbb{N}^{\ast}}\set(:A\in\mathcal{B}^{0}|{\set[2](:x\in
A|\overline{H}_{a,\frac{1}{p}}\left(  x\right)  <\infty:)}~\text{is not
meager}:)\\
&  =\bigcup_{p\in\mathbb{N}^{\ast}}\bigcup_{q\in\mathbb{N}^{\ast}}%
\set(:A\in\mathcal{B}^{0}|D_{a,\frac{1}{p},q}^{A}~\text{is not meager}:)\\
&  =\bigcup_{p\in\mathbb{N}^{\ast}}\bigcup_{q\in\mathbb{N}^{\ast}}%
\bigcup_{r\in\mathbb{N}^{\ast}}\mathcal{N}_{p,q,r}\text{,}%
\end{align*}
where%
\[
\mathcal{N}_{p,q,r}\overset{\mathrm{def}}{=}\set(:A\in\mathcal{B}^{0}|\exists
x\in A~\text{s.t.}~\bar{B}^{A}\left(  x,\frac{1}{r}\right)  \subset
D_{a,\frac{1}{p},q}^{A}:)\text{.}%
\]
One proves that the sets $\mathcal{N}_{p,q,r}$ are closed in the same way as
for $\mathcal{M}_{p,q,r}$. By Lemma \ref{LDCP}, for any $\varepsilon>0$,
arbitrarily close to any surface of $\mathcal{M}_{p,q,r}$ is a $\kappa
$-polyhedron $P$, whose vertices form an $\varepsilon$-net. Arbitrarily close
to this polyhedron is a smooth surface with a region of Gaussian curvature
more than $q$ close to each vertex of $P$ (see the proof of Lemma \ref{LDR}).
Hence $\mathcal{N}_{p,q,r}$ has empty interior in $\mathcal{B}^{0}$,
$\mathcal{N}$ is meager in $\mathcal{B}^{0}$, and consequently in
$\mathcal{B}$, too.
\end{proof}

We close the paper with a new result in convex geometry, whose proof is a
simplified version of the proof of Theorem \ref{g}. Notice that the
Gromov-Hausdorff topology differs from the Pompeiu-Hausdorff one (usually
considered in the framework of convexity), so Theorem \ref{Cor} cannot be
directly obtained from Theorem \ref{g}.

\begin{thm}
\label{Cor} On most convex surfaces $S\in\mathcal{S}$, at most points $x\in
S$, $\underline{K}(x)=\underline{G}(x)=0$ and $\overline{K}(x)=\overline
{G}(x)=\infty$.
\end{thm}


\bigskip

\noindent\textbf{Acknowledgement.} The authors are indebted to Tudor
Zamfirescu for suggesting the topic of this paper and for very useful
discussions, especially concerning the curvatures of convex surfaces.
Thanks are due to the referee, for useful remarks.

The second and third authors were supported by the grant
PN-II-ID-PCE-2011-3-0533 of the Romanian National Authority for Scientific
Research, CNCS-UEFISCDI.


\bigskip

Jin-ichi Itoh

\noindent{\small Faculty of Education, Kumamoto University \newline Kumamoto
860-8555, JAPAN \newline j-itoh@gpo.kumamoto-u.ac.jp}

\medskip

J\"oel Rouyer

\noindent{\small Institute of Mathematics ``Simion Stoilow'' of the Romanian
Academy, \newline P.O. Box 1-764, Bucharest 70700, ROMANIA \newline
Joel.Rouyer@ymail.com, Joel.Rouyer@imar.ro}

\medskip

Costin V\^{\i}lcu

\noindent{\small Institute of Mathematics ``Simion Stoilow'' of the Romanian
Academy, \newline P.O. Box 1-764, Bucharest 70700, ROMANIA \newline
Costin.Vilcu@imar.ro}

\end{document}